\documentclass[12pt]{amsart}
\usepackage{amssymb}

\usepackage[all]{xy}

\theoremstyle{plain}
\newtheorem{prop}{Proposition}
\newtheorem{thm}[prop]{Theorem}
\newtheorem{coro}[prop]{Corollary}
\newtheorem{lemm}[prop]{Lemma}
\newtheorem{conj}[prop]{Conjecture}

\newtheorem*{thm-int}{Theorem}
\theoremstyle{remark}
\newtheorem{exam}[prop]{Example}
\newtheorem{rema}[prop]{Remark}
\newtheorem{ques}[prop]{Question}

\theoremstyle{definition}

\def\Aut{{\rm Aut}}

\def\no{\noindent}

\def\cM{{\mathcal M}}

\def\lra{\longrightarrow}
\def\ra{\rightarrow}

\def\C{{\mathbb C}}

\def\F{{\mathbb F}}

\def\P{{\mathbb P}}
\def\Q{{\mathbb Q}}

\def\Z{{\mathbb Z}}
\def\C{{\mathbb C}}
\def\N{{\mathbb N}}

\def\bP{{\mathbb P}}

\def\SL{{\rm SL}}
\def\GL{{\rm GL}}
\def\PGL{{\rm PGL}}
\def\Sym{{\rm Sym}}

\newenvironment{dedication}
  {
   \thispagestyle{empty}
   \itshape             
   \raggedleft          
  }
  {\par 
   
  }

\begin{document}
\title[Torsion points of elliptic curves]{Torsion of elliptic curves and unlikely intersections}

\author{Fedor Bogomolov}
\address{Courant Institute of Mathematical Sciences, N.Y.U. \\
 251 Mercer str. \\
 New York, NY 10012, U.S.A.}
\address{National Research University Higher School of Economics, 
Russian Federation \\
AG Laboratory, HSE \\ 
7 Vavilova str., Moscow, Russia, 117312}

\email{bogomolo@cims.nyu.edu}

\author{Hang Fu}
\address{Courant Institute of Mathematical Sciences, N.Y.U. \\
 251 Mercer str. \\
 New York, NY 10012, U.S.A.}
\email{fu@cims.nyu.edu}

\author{Yuri Tschinkel}
\address{Courant Institute of Mathematical Sciences, N.Y.U. \\
 251 Mercer str. \\
 New York, NY 10012, U.S.A.}
\address{Simons Foundation\\
160 Fifth Av. \\
New York, NY 10010, U.S.A.}
\email{tschinkel@cims.nyu.edu}        


\begin{abstract}
We study effective versions of unlikely intersections of images of 
torsion points of elliptic curves on the projective line.  
\end{abstract}

\maketitle

\begin{dedication}
To Nigel Hitchin, with admiration. 
\end{dedication}

\setcounter{section}{0}       
\section*{Introduction}
\label{sect:introduction}

Let $k$ be a field of characteristic $\neq 2$ and $\bar{k}/k$ an algebraic closure of $k$.  
Let $E$ be an elliptic curve over $k$, presented as a double cover 
$$
\pi: E\ra \P^1,
$$ 
ramified in 4 points, and 
$E[\infty]\subset E(\bar{k})$ the set of its torsion points.  In \cite{BT-small} we proved:  

\begin{thm}
If $E_1, E_2$ are nonisomorphic elliptic curves over $\bar{\Q}$, then 
$$
\pi_1(E_1[\infty]) \cap \pi_2(E_2[\infty])
$$
is finite. 
\end{thm}

Here, we explore effective versions of this theorem, specifically, 
the size and structure of such intersections (see \cite{zannier} for
an extensive study of related problems).  
We expect the following universal bound:

 \begin{conj}[Effective Finiteness--EFC-I]
 \label{conj:main}
 There exists a constant $c>0$ such that 
 for every pair of nonisomorphic elliptic curves $E_1,E_2$  over $\C$ we have
 $$
  \pi_1(E_1[\infty])\cap \pi_2(E_2[\infty])<c.
 $$
 \end{conj}

We say that two subsets of the projective line
$$
S=\{ s_1,\ldots, s_n\},\quad S':= \{ s_1',\ldots, s_n'\} \subset \P^1(\bar{k})
$$
are projectively equivalent, and write $S\sim S'$, if there is a $\gamma\in \PGL_2(\bar{k})$
such that (modulo permutation of the indices) $s_i=\gamma(s_i')$, for all $i$.

Let $E$ be an elliptic curve over $k$, $e\in E$ the identity, and 
$$
\begin{array}{rcc}
E &  \stackrel{\iota}{\lra} & E\\
x & \mapsto & -x
\end{array}
$$
the standard involution. The corresponding quotient map
$$
\pi: E\ra E/\iota =\bP^1
$$
is ramified in the image of the 2-torsion points of $E(\bar{k})$. 
Conversely, for
$$
r:=\{ r_1,r_2,r_3, r_4\} \subset \bP^1(\bar{k}),
$$ 
the double cover 
$$
\pi_r: E_r\ra \bP^1
$$ 
with ramification in $r$ defines an elliptic curve; 
given another such $r'$, the curves $E_r$ and $E_{r'}$ are isomorphic (over $\bar{k})$ 
if and only if $r\sim r'$, in particular, the image of 2-torsion determines the elliptic curve, up to isomorphism.

Let $E_r[n]\subset E_r(\bar{k})$ be the set of elements of order {\em exactly} $n$, for $n\in \N$. 
The behavior of torsion points of other small orders is also
simple: 
$$
\pi_r(E_r[3])\sim \{ 1,\zeta_3,\zeta_3^2,\infty\}, 
$$ 
where $\zeta_3$ is a nontrivial third root of 1, and 
$$
\pi_r(E_r[4]) \sim \{ 0,1,-1,i,-i,\infty \}. 
$$ 
In particular, up to projective equivalence, these are {\em independent} of $E_r$.  
However, for all $n\ge 5$, 
the sets  $\pi_r(E_r[n])$, modulo $\PGL_2(\bar{k})$, do depend on $E_r$, and it is 
tempting to inquire into the nature of this dependence. 

In this note, we study $\pi_r(E_r[n])$, for varying curves $E_r$ and varying $n$. 
Our  goal is to establish effective and uniform finiteness results for intersections
$$
\pi_r(E_r[n])\cap \pi_{r'}(E_{r'}[n']), \quad n, n'\in \N,
$$
for elliptic curves $E_r,E_{r'}$, defined over $k$. 
We formulate several conjectures in this direction and provide evidence for them.

The next step is to ask: given elliptic curves $E_r,E_{r'}$ over $\bar{k}$, 
when is 
$$
r \subset \pi_{r'}(E_{r'}[\infty])?
$$ 
We modify this question as follows:
Which minimal subsets $\tilde{L}\subset \P^1(\bar{k})$ have the property
$$
r\subset \tilde{L} \quad \Rightarrow \pi_r(E_r[\infty])\subseteq \tilde{L}?
$$ 
The sets $\tilde{L}$ carry involutions, obtained from 
the translation action of the 2-torsion points of $E$ on $E$, 
which
descends, via $\pi$, to an action on $\P^1$
and defines an embedding of $\Z/2\oplus \Z/2\hookrightarrow \PGL_2(\bar{k})$.
It is conjugated to the standard embedding
of $\Z/2\oplus \Z/2$, generated by involutions 
$$
z\mapsto -z\quad \text{  and } \quad z\mapsto 1/z,
$$ 
acting on $\tilde{L}$.  
This observation is crucial for the discussion in Section~\ref{sect:cyclo}, where we prove
that, modulo projectivities, $L:=\tilde{L}\setminus \{ \infty\} $ are fields.

\bigskip

\noindent {\bf Acknowledgments:} 
The first author was partially supported by the Russian Academic
Excellence Project `5-100' and by Simons Fellowship  and
by  EPSRC programme grant EP/M024830. 
The second author was supported by the MacCracken Program offered
by New York University.
The third author was partially supported by NSF grant 1601912.

\section{Generalities}
\label{sect:general}


Let $j:\mathcal E\ra \bP^1$ be the standard universal elliptic curve, with $j$ the $j$-invariant morphism. 
Consider the diagram

\

\centerline{
\xymatrix{ 
E_{\lambda} \ar^{\iota}[r] 
\ar@{}[d]|-*[@]{\subset}   & P_{\lambda} \ar@{}[d]|-*[@]{\subset}
\\ 
\mathcal E \ar^{\iota}[r] \ar[d]_{j}  &   \mathcal P \ar[d]_j  \\
 \bP^1 \ar@{=}[r]  & \bP^1
}
}

\

\noindent
assigning to each fiber $E_{\lambda}:=j^{-1}(\lambda)$ the quotient $P_{\lambda}= \pi(E_{\lambda})\simeq \bP^1$, 
by the involution $\iota: x\mapsto -x$ on $E_{\lambda}$. (This is well-defined even for singular fibers of $j$.)

Note that $\mathcal P\ra \bP^1$ is a $\PGL_2$-torsor. Taking fiberwise $n$-symmetric product: 
$$
 P_{\lambda} \mapsto  \Sym^n(P_{\lambda}) 
 $$
 we have associated $\PGL_2$-torsors  
 $$
 j_n: \mathcal P_n=\Sym^n(\mathcal P)\ra \bP^1. 
 $$
Taking $\PGL_2$-invariants, we have a canonical projection 
$$
\Sym^n(P_\lambda) \ra \mathcal M_{0,n}(P_{\lambda}) \simeq \mathcal M_{0,n},
$$
to the moduli space of $n$-points on $\bP^1$. 
The associated $\PGL_2$-torsor is trivial; fixing a trivialization we obtain a morphism
$$
\mu_n: \mathcal P_n \ra \overline{\mathcal M}_{0,n}
$$

For every $N\in \N$, we have the modular curve $X(N)\ra \bP^1$, parametrizing pairs of elliptic curves together with $N$-torsion subgroups.
The involution $\iota$ induces an involution on every $X(N)$, we have the induced quotient 
$$
X(N)\ra Y(N):=X(N)/\iota.
$$ 
Since the family $j:\mathcal E\ra \P^1$ has maximal monodromy $\SL_2(\Z)$, the curves $X(N)$ and $Y(N)$ are irreducible. 
We have a natural embedding $Y(N)\hookrightarrow \mathcal P$. Put
$$
Y:=\cup_{N\in \N} Y(N)
$$
and consider 
$$
\Sym^n(Y)\hookrightarrow \mathcal P_n \ra  \overline{\mathcal M}_{0,n}.
$$
Note that $\Sym^n(Y)$ is a union of infinitely many irreducible curves, each corresponding to 
an orbit of the action of the monodromy group $\PGL_2(\Z)$ on the generic fiber of the restriction of $j_n$ to $\Sym^n(Y)$.
Let $Y_{n,\omega}\subset \Sym^n(Y)$ 
be an irreducible component corresponding to a $\PGL_2(\Z)$-orbit $\omega$ (for the monodromy action, as above). 
We now formulate conjectures
about $\mu_n$, for small $n$, which guide our approach to the study of images of torsion points. 

\begin{conj}
\label{conj:1}
The map 
$$
\mu_4 : Y_{4,\omega} \to \overline{\mathcal M}_{0,4} =\P^1
$$ 
is finite surjective, for all but finitely many $\omega$. 
\end{conj}

\begin{conj}
\label{conj:2}
The map 
$$
(\mu_4,j):  Y_{4,\omega}\to \overline{\mathcal M}_{0,4}\times \P^1
$$ 
is
a rational embedding, for all but finitely many $\omega$. 
\end{conj}

\begin{conj}
\label{conj:3}
The map 
$$
\mu_5 : Y_{5,\omega }\to \overline{\mathcal M}_{0,5}
$$ 
is a rational
embedding, for all but finitely many $\omega$. Moreover, if for some distinct orbits $\omega$ and $\omega'$ 
the corresponding images $\mu_5(Y_{5,\omega })$ and  $\mu_5(Y_{5,\omega'})$ are curves, then they are different. 
\end{conj}

\begin{conj}
\label{conj:4}
The map 
 $$
 \mu_6 : Y_{6,\omega}\to \overline{\mathcal M}_{0,6}
 $$
is a rational embedding, for all but finitely many $\omega$. 
Moreover, if $\mu_6 (Y_{6,\omega})$ is a curve then there exist 
at most finitely many $\omega'$ such that 
\begin{itemize}
\item $\mu_6 (Y_{6,\omega'})$ is a curve and 
\item $\mu_6(Y_{6,\omega})\cap \mu_6 (Y_{6,\omega'})\neq \emptyset$.
\end{itemize}
\end{conj}

\section{Examples and evidence}
\label{sect:ex}

We now discuss examples and evidence for Conjectures in Section~\ref{sect:general}.

 \begin{exam}
We have
\begin{itemize}
\item  $\mu_4(\Sym^4(Y(2))\simeq \overline{\mathcal M}_{0,4}=\bP^1$,
\item $\mu_4(\Sym^4(Y(3))$ is a point in  $\overline{\mathcal M}_{0,4}$.
\end{itemize}
Consider $\Sym^4(Y(4))$. Note that $\pi(E[4])=\{ 0,1,-1, i,-i,\infty\} $
is an orbit of the symmetric group  $\mathfrak S_4$, acting on $\P^1$. 
The pairs 
$$
(0, \infty),(1,-1),(i,-i)
$$
are pairs of stable points for $3$ even involutions in $\mathfrak S_4$,
and the action of $\mathfrak S_4$ is transitive on pairs and inside
each pair. There are two different $\mathfrak S_4$-orbits
of $4$-tuples: either the orbit contains two pairs
of vertices such as $(0, \infty),(1,-1)$, or a pair and two points
from different pairs $(0, \infty),(1, i)$.
Thus $\Sym^4(Y(4))$ has two components which project to different
points modulo $\PGL_2$; therefore, there exist 
exceptional orbits $\omega$ such that  $\mu_4(Y_{4,\omega})$ is a point. 
\end{exam}


\begin{lemm}
\label{lemm:mu-cn} 
If $\mu_4(Y_{4,\omega})$ is a point
then all cross ratios of 4-tuples of points parametrized by $Y_{4,\omega}$
are constant. 
\end{lemm} 
 
\begin{proof}
The map $\mu_4 $ can be viewed
as a composition
$$
 (\P^1)^4\stackrel{\mathit{cr}}{\longrightarrow} (\Z/2\oplus  \Z/2)   \backslash  (\P^1)^4/\PGL_2 = \P^1_1\to  \mathfrak S_3 \backslash \P^1_1.
$$
Thus we have a diagram

\centerline{
\xymatrix{
(\P^1)^4 \ar[r]^{cr\quad\quad\quad }\ar[d]           &   (\Z/2\oplus  \Z/2) \backslash (\P^1)^4 /\PGL_2 \ar[d]^{\mathfrak S_3} \\
 \mathfrak S_4\backslash (\P^1)^4 \ar[r] &  \mathfrak S_4\backslash (\P^1)^4/\PGL_2
 }
 }

\

\no
Note that any irreducible $Y_{4,\omega}$ lifts
to a union of connected components $Y_{4,\omega,i}\subset (\Z/2\oplus  \Z/2)\backslash Y^4$, 
where cross-ratio is well defined.
Thus if $\mu_4$ is a rational function of cross-ratio
on any four-tuple of points and if $\mu_4$ is constant
then the cross-ratio is also constant.
\end{proof}

\begin{prop}
\label{prop:sur}
There exist orbits $\omega$ such that  
$$
\mu_4: Y_{4,\omega}\ra \P^1
$$
is surjective.   
\end{prop}

\begin{proof}
The singular fiber $\mathcal E_{\infty}:=j^{-1}(\infty)$ 
is an irreducible rational curve with one node $p_{\infty}$. The group scheme $\cup _{d\mid n} \mathcal E[d]$, whose generic fiber is isomorphic to 
$\Z/n\oplus\Z/n$, specializes to $\{\zeta_n^i\}\subset \mathbb G_m=\mathcal E_{\infty}\setminus p_{\infty}$. 
Let $\mathcal E_{\infty}[n]$ be the specialization of $\mathcal E[n]$; then 
\begin{itemize}
\item $\mathcal E_{\infty}[n]\subset \{ \zeta_n^i\}$, 
\item there exists a subgroup scheme $\mathcal W_n\simeq \Z/n\subset\Z/n\oplus \Z/n$ in the group scheme 
of points killed by $n$, specializing to  $\mathcal E_{\infty}$, while the complemenary branches specialize to 
 $p_{\infty}$. 
\end{itemize} 
Taking the quotient by $\iota$, we find that
 $((\Z/n\oplus \Z/n)\setminus \Z/n)/\iota $ specializes to $0$ in the fiber $\P^1_{\infty}$ and
 all other points specialize to subset in $(\Z/n)/\iota$; the limit depends
on the selected direction of specialization.

Assume that we have distinct points $\{ z_1,z_2,z_3,z_4\} \subset \pi(E[n])$, for a smooth fiber $E$ of $\mathcal E$, such that 
$$
z_1,z_2\in  W_n/\iota \quad \text{ and } \quad z_3,z_4\notin W_n/\iota.
$$ 
The $z_1, z_2$ can be specialized to different nonzero points in $\mathcal E_{\infty}/\iota$, and 
$z_3,z_4$ will specialize to $0$. 

Assume that $\mu_4$ is constant, i.e., the cross-ratio is constant. 
Since  $z_3,z_4$ will specialize to $0$, the cross-ratio equals 1.  
Then
$$
(z_1- z_3)(z_2-z_4)= (z_2- z_3)(z_1-z_4), 
$$
and 
$$
z_1(z_3-z_4)=z_2(z_3-z_4). 
$$
Near the special fiber, $z_3\neq z_4$, thus $z_1=z_2$, contradiction. 
Thus on orbits of this type, $\mu_4$ is not constant, hence surjective.
 \end{proof}
 


\section{Geometric approach to effective finiteness} 
\label{sect:geom}

Let $E:=E_r,E':=E_{r'}$ be elliptic curves.
Consider the diagram

\centerline{
\xymatrix{
C\ar[d] \ar[r]&  E\times E' \ar[d] \\
\Delta \ar[r] &    \P^1\times \P^1
}
}

\no
where 
$C\subset E\times E'$ be the fiberwise product over the diagonal  
$\Delta\subset \P^1\times \P^1$. If $r\neq r'$ then $C$ has genus $\ge 2$. 
By Raynaud's theorem \cite{raynaud}, 
$$
C(\bar{k})\cap E[\infty]\times E'[\infty] 
$$
is finite, since it is 
the preimage of $\pi(E[\infty])\cap\pi(E'[\infty])\subset \Delta$, the latter set is also finite. 
This finiteness argument appeared in \cite{BT-small}. 

Consider the curves $C$ occurring in this construction. We have a diagram

\centerline{
\xymatrix{ 
C \ar[d]_{\sigma'}\ar[r]^{\sigma} & E\\
E' &  
}
}

\

\noindent
where $\sigma, \sigma'$ are involutions with fixed points $c_1,c_2$ and $c_1', c_2'$, respectively.  Assume that 
$$
r\cap r' = \{ 0, 1, \infty\}. 
$$
Then the product involution $\sigma\sigma'$ on $C\subset E\times E'$ has fixed points in the 6 
preimages of the points $\{ 0,1,\infty \}\subset \Delta_{\P^1}\subset \P^1\times \P^1$ (diagonally), i.e., is the hyperelliptic involution.  
Thus there is an action of $ \Z/2\oplus \Z/2$ on $C$, induced by the covering maps $\pi$ and $\pi'$.  
The curve $C\subset A=E\times E'$ has self-intersection $C^2=8$
since it is a double cover of both  $E$ and $E'$ and its class is equal to $2(E+E')$.  

\begin{itemize}
\item
If the genus $\mathsf g(C)=2$ (three such points) then the image of $C$ in its Jacobian $J(C)$ has self-intersection 2. 
Consider the map
$$
\nu: J(C)\to  A= E\times E'.
$$ 
and let $n$ be its degree.  
The preimage $\nu^{-1}(C)\subset J(C)$ has self-intersection $8n$. 
On the other hand, its homology class is equal to $n$ translations of $C$, hence has self-intersection $2n^2$, thus $n=4$. 
Moreover, $\ker(\nu) = \Z/2\oplus \Z/2$, generated by the pairwise differences of preimages of points $\{ 0,1,\infty \}$.  
Thus, $J(C)$ is 4-isogenous to $A:=E\times E'$ and 
$\nu(C)$ is singular, with nodes exactly at the preimages of  $\{ 0,1,\infty \}\subset \Delta_{\P^1}$. 
Consider a point $c\in C\subset J(C)$ and assume that $\nu(c)$ has order $m$ with respect to  $0\in A$. 
Then $c$ has order $m$ or $2m$ in $J(C)$, with respect to $0 \in J(C)$.  
Hence the corresponding curve $Y(m)\subset \P^1\times \P^1$ (viewed as a moduli
space of pairs $E,E'$)
is given as an intersection of genus $2$ curves containing a point of order $m$ or $2m$, respectively. This is a locus in the moduli space
 $\mathcal M_2$ of genus 2 curves. 
\item 
If $\mathsf g(C)=3$ (two such points) then
there are three quotients of $C$ which are
elliptic curves $E_1,E_2,E_3$, with involutions $\sigma_i\in \Z/2\oplus \Z/2$ 
fixing $4$ points on $E_i$ which are invariant under the hyperelliptic
involution given by complement to $\sigma_j$. 
The kernel of
$$
\nu_i: J(C)\to E_j\times E_k
$$ 
contains
$E_i$, for $i,j,k\in \{ 1,2,3\} $. 
\item 
If $\mathsf g(C)= 4$ then $C$ is $C/\sigma_i = E_i, i=1,2$
and $C/\sigma_1\sigma_2 = C'$ where $\mathsf g(C')=2$ and
there are exactly two ramification points on $C'$.
\item 
If $\mathsf g(C)= 5$ then $C/\sigma_1\sigma_2 = C'$ is a hyperelliptic
curve of genus $3$ and the covering is an unramified double cover.
\end{itemize}

\begin{rema}
\label{rema:stand}
Assume that there is $b\in \P^1$ 
and a subset $S\subset C(\bar{k})$ such that $S+b \subset C\subset E\times E'$.
 Then 
 $$
 \# S\le 8 = C^2 = C \cap (C + b);
 $$ 
 hence we have
  at most $8$ points $c_i\in \P^1$ such that  for $x$-coordinates
 $ c_i+_1 b = c_i +_2 b$, where the summation $+_1$ corresponds to the
 summation on the first curve and $+_2$ on the second.
\end{rema}

\begin{rema}
The construction can be extended to products of more than two elliptic curves. 
We may consider
$$
\pi:=\prod_{i=1}^r   \pi_i :  \mathcal A:=\prod E_i\to \mathcal P :=\prod \P^1_i.
$$
The ramification divisor of $\pi :\mathcal A \ra \mathcal P$ 
is a union of products of projective lines. Let $\Delta=\P^1\subset \mathcal P$ be the diagonal,
there exists canonical identifications  $\delta_i: \P_i^1 \simeq \Delta$.  
If $p\in \Delta$  is contained in  $\delta_i(\pi_i(E_i[\infty]))$, for all $i$, 
then the preimage of $p$ in $\mathcal A$ is contained 
in the preimage of the diagonal. This is a curve of genus at least 2, provided 
there exist $E_i, E_j$ with $r_i\neq r_j$.  Then the set of such $p$ is finite. 
In particular, if $E$ is defined over a number field $k$ and $p$ is defined over a proper subfield, then 
$p$ is also in the image torsion points of $\gamma(E)$, where $\gamma$ is a Galois conjugation. Hence, 
the existence of torsion points with $x$-coordinate in a smaller field has a geometric implication. 
\end{rema}

We expect the following version of Conjecture~\ref{conj:main}:

\begin{conj}[Effective Finiteness--EFC-II]
\label{conj:main2}
There exists a constant $c>0$ such that for every elliptic curve $E_r$ over a number field 
and every $\gamma\in  \PGL_2(\bar{\Q})$ with $\gamma(r)\neq r$ we have
 $$
\pi_r(E_r[\infty])\cap \pi_\gamma(E_\gamma[\infty]) < c.
$$
\end{conj}

\section{Fields generated by elliptic division}
\label{sect:cyclo}

In this section, we explore properties of subsets of $\P^1(\bar{k})$ generated by images of torsion points, following closely \cite{BT-small}.
For 
$$
r:=\{ r_1, r_2,r_3,r_4\}\subset \P^1(\bar{k}),
$$ 
a set of four distinct points, let $E_r$ be the corresponding elliptic curve defined in the Introduction. Let
$$
\tilde{L}_r\subset \P^1(\bar{k})
$$ 
be the smallest subset such that for every $E_{r'}$ with $r'\subseteq \tilde{L}_r$ we have $\pi_{r'}(E_{r'}[\infty])\subseteq \tilde{L}_r$.

\begin{thm}\cite{BT-small}
\label{thm:field}
Let $k$ be a number field. 
For every $a\in k\setminus \{0,\pm1,\pm i\}$, and 
\begin{equation}
\label{eqn:ra}
r=r_a:=\{ a,-a,a^{-1} , -a^{-1}\} \subset \P^1(k)
\end{equation}
the set
$$
L_{a}:=\tilde{L}_{r_a}\backslash\{\infty\}
$$ 
is a field.
\end{thm}

At first glance, it is rather surprising that such a simple and natural construction, inspired by comparisons of $x$-coordinates of torsion points of
elliptic curves, produces a field. The conceptual reason for this is the rather peculiar structure of 4-torsion points of elliptic curves:
translations by 2-torsion points yields, upon projection to $\P^1$, {\em two} standard commuting involutions on $\P^1(\bar{k})$, which allow to define
addition and multiplication on $L_a$.  

We may inquire about arithmetic and geometric properties of the fields $L_a$. For $a\in \bar{k}$ we let $k(a)\subseteq \bar{k}$
denote the smallest subfield containing $a$. We have:
\begin{itemize}
\item 
For every $a\in \bar{k}$, the field $L_a$ is a Galois extension of $\Q(a)$. 
\item 
For every $k$ of characteristic zero, $L_a$ contains $\Q^{ab}$, the maximal abelian extension of $\Q$. 
\item
The field $L_\zeta$, where $\zeta$ is a primitive
root of order $8$, is contained in any field $L_a$.
Indeed, the corresponding elliptic curve $E$ has ramification subset 
$$
\{ \zeta,\zeta^3,\zeta^5,\zeta^7\}, 
$$ 
which is projectively equivalent to $\{ 1,-1, i,-i\} \subset \pi(E[4])$.
Since $\pi(E[4])$ projectively does not depend on the curve $E$, we obtain
that $L_\zeta\subset L_a$, for all $a$.
The same holds for $L_a$ where $E_a$ is isomorphic
to $E_3$ (elliptic curve with an automorphism of order $3$).
\item 
The field $L_a$ is {\em contained} in a field obtained as an iteration of Galois extensions with Galois groups either abelian or 
$\PGL_2(\mathbb F_q)$, for various prime powers $q$. Is $L_a$ equal to such an extension? 
As soon as the absolute Galois group is not equal to a group of this type, e.g., for
a number field $k$, we have 
$$
L_a\subsetneq \bar{k}.
$$ 
\item 
Let $a,a'\in \bar{\Q}$ be algebraic numbers such that $\Q(a)=\Q(a')$. Then $L_a=L_{a'}$. 
Varying $a\in \bar{\Q}$, we obtain a supply of interesting infinite extensions $L_a/\Q$.
\end{itemize}

The rest of this section is devoted to the proof of Theorem~\ref{thm:field}. 

\begin{proof}
Let $r_0:=\{ 0, \infty, 1, -1\}$ and put $L:=\tilde{L}_{r_0}\setminus \{ \infty\}$. Let 
$$
\pi=\pi_{r_a}:E_{r_a}\ra \P^1
$$ 
be the elliptic curve
with ramification in $r_a$. Since 
$$
\{0,\infty,\pm1\}\subseteq\pi(E_{r_a}[4]),
$$
we have 
$L\subseteq L_{a}$, for all $a$. 
We first show that $L$ is a field.

\

{\em Step 1.} 
$L\backslash\{0\}$ is a multiplicative group.
Indeed, for any $b\in L\backslash\{0\}$, we have
$$
r_0:=\{ 0,1,-1,\infty\}  = b^{-1} \cdot \{0,b,-b,\infty\} =:r_b
$$
and hence 
$$
L_{r_b}= b\cdot L_{r_0}= b\cdot L.
$$
Since $b^{-1},-b^{-1}\in L$ we also
have $\{ 0,1,-1,\infty\} \subset b\cdot L$.
Thus $L\subseteq b L$. 
Similarly, $ L\subseteq b^{-1}\cdot L$ or $b \cdot L\subseteq L$,  
which implies $L=bL$.
Thus for any $a,b\in L$ we have $ ab \in L$,
and since the same holds for $ ab^{-1}$, $b\neq 0$, we obtain
$L\setminus \{ 0\} \subseteq \bar{k}^\times$.

\

{\it Step 2.} 
Let 
$$
\Aut_{L}:=\{ \gamma \in \PGL_2(\bar{k}) \,  | \,  \gamma(\tilde{L})\subseteq \tilde{L} \}
$$
be the subgroup preserving $\tilde{L}$. 
It is nontrivial, since it contains $L\setminus \{ 0\} $ as a multiplicative subgroup, together 
with the involution $x\mapsto x^{-1}$. 
Consider 
$$
\gamma_1: x\mapsto (x-1)/(x+1).
$$
It is an involution with
$\gamma_1(\infty)= 1, \gamma_1(0)= -1$
and hence $\gamma_1$ is coming from
$r:=\{0,1,-1,\infty\}$.
Thus it maps $L$ into $L$ and
$\gamma_1\in \Aut_L$

Consider any pair of distinct elements
$\{ b,c\} \subset L$: it  can be transformed into $\{ 0,1\}$
by an element from $\Aut_L$.
If $b\neq 0,\infty$ then,
dividing on $b$, we obtain
$\{ 1,c/b\}$ and $\gamma_1(\{ 1,c/b\})=\{0,1\}$.
If $b=0$ and $c\neq \infty$ then, dividing
on $c$, we obtain $\{0,1\}$.
If $b=0,c= \infty$ then $\gamma_1(\{0,\infty\})= \{-1,1\}$
and we reduce to the first case.

\

{\it Step 3.}
$L$ is closed under addition.
We show that $\gamma: x\mapsto x+1$ is contained in
$\Aut_L$: by Step 2, there exists a  $g\in \Aut_L$ 
which maps $\{-1,\infty\}$ to $\{0,\infty\}$ and hence
$\{-1,0,\infty\}$ to $\{0,b,\infty\}$, 
for some $b\in L\setminus \{0\}$.
Then $b^{-1}g\in \Aut_L$ maps  $\{ -1,0,\infty\}$ to
$\{0,1,\infty\}$ and hence $b^{-1}g(x)= \gamma(x)=x+1$.
Thus for any $a\in L$ we have $a+b = b(a/b + 1)\in L$, which shows that 
$L$ is an abelian group.

\

Now let us turn to the general $L_{a}$.

\

{\it Step 4.} 
Note that $L\subset L_a$ and that 
$L_a$ is closed under taking square roots.
Indeed for any $a\in L$ and $E_r$ with $r:=\{ 0,1,a,\infty\}$, we have 
$\sqrt{a}\in\pi_r(E_r[4])$ and hence $\sqrt{a}\in L_a$.
Furthermore, for any $a,b\in L_a$ we have $\sqrt{ab}\in L_a$.
Indeed, consider the curve $E_r$ with $r= \{0,a,b,\infty\}$. 
Then  $\sqrt{ab}\in \pi(E_r[4])$, since 
the involution $z\to ab/z$ is contained in the subgroup $\Z/2\oplus \Z/2$ corresponding to the two-torsion on $E_r$, 
its invariant points are in $\pi_r(E_r[4])$.
Iterating, we obtain that 
$$
\sqrt[2^{m-1}]{b_1\cdots b_m}\in \tilde{L}_a\setminus \{ \infty\} \text{  for all } \, b_i\in 
\tilde{L}_a\setminus \{ \infty\}
$$

\

{\em Step 5.}
For all $b\in L_a, c\in L$ we have $\sqrt{b+ c}\in L_a$.
Indeed, for $c\in L$ we know that there is a solution $d\in L$
of the quadratic equation $d^2+d + c =0$.
Consider the curve $E_r$ for $r:=\{\infty,b, d , d+1\}$.
Then 
$$
d\pm \sqrt{b-d}\in \pi(E_r[4])
$$
and hence $d\pm\sqrt{b-d}\in L_a$.
Thus
$$
\sqrt{(\sqrt{b-d} + d)(\sqrt{b-d}- d)}= \sqrt{b -d^2 -d}= \sqrt{b+c} \in L_a.
$$

\

{\em Step. 6.}
Let $P_m\in L[x]$ be a monic
polynomial of degree $m$ and let $b\in L_a$.
Then there is an $N(m)\in \N$ such that 
$$
\sqrt[4^{N(m)}]{P_m(b)}\in L_a.
$$
Indeed, we have
$$
P_m(b)= c_m + b(c_{m-1} + b(c_{m-2}+\cdots )\cdots ). 
$$
The statement holds for $m=1$ by Step 4.
Assume that it holds for $m-1$.
Then $c_{m-1} + b(c_{m-2}+ \cdots )= d^{4^{N(m-1)}}$ for
some $d\in L_a$.
We can then write 
$$
P^m(b)=c_m + bd^{4^{N(m-1)}},
$$
by taking $t= \sqrt[4^{N(m-1)}]{b}$ and $u_m= \sqrt[4^{N(m-1)}]{c_m}$
we obtain 
$$
P^m(b)= \prod (t+ \zeta^i u_m),
$$ 
where
$t\in L_a, u_m \in L$ and $\zeta^i$ runs through the roots of unity
of order $4^{N(m-1)}$.

By Steps 4 and 5, we obtain that $4^m 4^{N(m-1)}$-th
root of  $P_m(b)$
is contained in $L_a$, thus the result holds for
$N(m)= 4^{N(m-1)}$

\

{\em Step 7.}
Let $b\in L_a$ be any algebraic element over $L$.
Then the field $L(b)$ is a finite extension of $L$ and
there is an $n\in \N$ such that any $x\in L(b)$
can be represented as a monic polynomial of $b$ with coefficients in $L$
of degree $ \leq n$. For such $n$ we define
a power $4^N$ such $\sqrt[4N]{x}\in L_a$, but then
any element in $L(b)$  is in $L_a$.

\end{proof}

\

\begin{rema}
In the proof we have only used
points in $\pi(E[4])$. Therefore, for any subset $D\subset \N$ 
containing $4$ we can define $L_{a,D}$, as the smallest subset containing all $\pi(E[n])$ for all $n\in D$ and
all elliptic curves obtained as double covers with ramification in $L_{a,D}$.
It will also be a field.

For example, if $D=\{3,4\} $ then $L_{a,D}$ 
is exactly the closure of $L_a$ under abelian degree $2$ and $3$ extensions, since
$\PGL_2(\mathbb F_2) = \mathfrak S_3$ and $\PGL_2(\mathbb F_3)= \mathfrak S_4$ and both groups
 are solvable with abelian quotients of exponent $3,2$.
 \end{rema}

 On $(\Sym^4 (\P^1(\bar \Q))\setminus \Delta)/ \PGL_2(\bar \Q)$ 
 we can define a directed graph structure $DGS$, postulating that  
 $$
 r_z= \{ z_1,z_2,z_3,z_4\} \ra r_w= \{ w_1,w_2,w_3,w_4\}
 $$
 if there is an elliptic curve $E'$ isogeneous to $E_{r_z}$ such that $r_w$ is projectively
equivalent to a subset in $\pi(E'[\infty])$.
Any path in the graph is equivalent to a path contained
in $(\Sym^4 (\P^1(L(E)))\setminus \Delta)/ \PGL_2(\bar \Q)$, for some $E$.
The graph contains cycles, periodic orbits, and preperiodic orbits, i.e., 
paths which at some moment end in  periodic orbits.

\begin{ques}
Consider the field $L_0= L_{r_0}$ for $r_0=\{ 0,1,-1,\infty\} $.
Does  
$$
(\Sym^4(\P^1(L(E)))\setminus \Delta)/ \PGL_2(L(E))
$$ 
consist
of one cycle in $DGS$?
Note that any path beginning from $r_0$
extends
to a cycle (in many different ways) since $r_0$
is $\PGL_2$-equivalent to a four-tuple of points of order $4$
on any elliptic curve.
\end{ques}

\begin{rema}
In Step 7, we have used algebraicity of $L_a/L$, and we do not know how to
extend the proof to geometric fields. What are the properties of $L_a$ in 
geometric situations, when $a$ is transcendental over $k$? 
\end{rema}

We have seen in the proof that the field $L_a$ is closed under extensions of degree 2. We also have: 

\begin{lemm}
For any $b\in L_a$, we have $\sqrt[3]{b}\in L_a$.
\end{lemm}

\begin{proof} 
Consider a curve $E_r$ with $r:= \{ b, \sqrt b, -\sqrt b, \infty\}$. 
Its 3-division polynomial takes the form:
$$
f_3(x) = 3x^4 - 4b x^3 - 6b x^2 + 12 b^2 x - 4 b^3 - b^2.
$$
We can represent it as a product:
$3\prod_{i=1}^4 (x-x_i)$, where the set $\{ x_i\} \subset L_a$ is equal
$\pi(E_r[3])$. The corresponding cubic resolvent
$$
rc(x):=\prod (x- (x_ ix_j+ x_k x_l)),
$$ 
where $(i,j),(k,l)$ is any splitting
into pairs of indices among $1,2,3,4$.
In terms of $b$, we have
$$
rc(x)= x^3 + 2bx^2 + 4b^2 x/3  + 8b^3/3 - 128 b^4/27 + 64 b^5/27.
$$
Since the set $\{ x_i\}$ is projectively equivalent to $\{ 0,1,\zeta_3,\zeta_3^2\}$, 
we can see that the cubic polynomial above
has the form $C (x^3 + B)$, for some constants $C,B$.
It can be checked that 
$$
rc( 2b(2x-1)/ 3 ) = (4b/3)^3 (x^3 + (b-1)^2).
$$
After a projective map in $\PGL_2(L_a)$ we can transform the
the elements $x_ix_j+ x_kx_l$ into $-\sqrt[3]{(b-1)^2}$.
Hence $-\sqrt[3]{(b-1)^2}\in L_a$, for any $b\in L_a$; since
 $L_a$ is a field  closed under $2$-extensions we obtain the claim. 
\end{proof}

This raises a natural
 
\begin{ques}
Is $L_a$ is closed under taking roots of arbitrary degree?
\end{ques}

If we add $\mathbb G_m$ to the set of allowed elliptic curves then the answer
is affirmative. However, there may exist a purely {\em elliptic} substitute
for obtaining roots of prescribed order.

\begin{coro}
If the $j(E)\in L_a$ then any set $\{ b,-b,b^{-1},-b^{-1}\} $ with $\mu_4((b,-b,b^{-1},-b^{-1}))=j(E)$
is contained in $L_a$. Note that such $b$ are solutions
of a cubic equation. Thus $L_a$ depends only
on the curve $E$ and we will write $L(E)$. 
\end{coro}

It is also easy to see that $L(E)=L(E')$ if $E$ and $E'$ are isogenous.

\section{Intersections}
\label{sect:inter}

In this section we present further results concerning intersections 
$$
\pi_1(E_1[\infty])\cap \pi_2(E_2[\infty])
$$ 
for
different  elliptic curves $E_1,E_2$ and
provide evidence for the Effective Finiteness Conjecture \ref{conj:main}.

  \begin{prop}
  \label{prop:ar}
 Assume that 
 \begin{equation}
 \label{eqn:eqb}
 \pi_1(E_1[4])=\pi_2(E_2[4])= \{ 0,1, -1, i, -i, \infty\} 
 \end{equation}
 and that 
 $$
\#\{  \pi_1(E_1[3])\cap \pi_2(E_2[3])\} \ge 2.
 $$
Then $r_1=r_2$ and $E_1 = E_2$. 
 \end{prop}

\begin{proof} 
By our assumption \eqref{eqn:eqb}, 
$E_i$ are given by the equation 
$$
y^2=x^4- t_ix^2 + 1.
$$
With $a_i$ defined by
$$
r_i = \{ a_i,-a_i,a_i^{-1},-a_i^{-1} \},
$$
we have
$$
t_i= a^2_i + a^{-2}_i.
$$
We assume that $\pi_i(e_i)=a_i$. 
In this case, points $\pi_i(E_i[3])\subset \bar{\Q}\subset \P^1$ are 
the roots of
\begin{equation}
\label{eqn:aaa}
x^4+2a x^3-(2/a)x-1=0 
\end{equation}
or, equivalently, 
$$
2x^{3}a^{2}+(x^{4}-1)a-2x.
$$
If $x,y\in\pi_{a_{1}}(E_{a_{1}}[3])\cap\pi_{a_{2}}(E_{a_{2}}[3])$,
where $x\neq y$ and $a_{1}\neq a_{2}$, then $a_{1}$
and $a_{2}$ are the roots of $2x^{3}a^{2}+(x^{4}-1)a-2x$
and of $2y^{3}a^{2}+(y^{4}-1)a-2y$, that means that their coefficients are proportional
\[
\frac{2x^{3}}{2y^{3}}=\frac{x^{4}-1}{y^{4}-1}=\frac{-2x}{-2y}.
\]
Then, on the one hand, 
$x^{3}/y^{3}=x/y$ implies $x^{2}=y^{2}$, and hence $x=-y$, by our assumption that $x\neq y$. 
On the other hand, 
$$
x/y=-1 = (x^{4}-1)/(y^{4}-1)=1,
$$
a contradiction.
\end{proof}

 Given any $x\in \bar{\Q}$ 
 we obtain $a_i=a_i(x)$, $i=1,2$, which satisfy \eqref{eqn:aaa}.
 Then the resulting elliptic curves $E_i$ satisfy
 \eqref{eqn:eqb} and we have
$$
\#\{  \pi_1(E_1[3])\cap \pi_2(E_2[3])\}  = 1.
 $$
 unless 
    $$
    (x^4-1)^2 + 16x^4 = x^8 + 14x^4 + 1= 0
    \quad \text{ or }\quad  x^4= -7 \pm 4\sqrt{3}.
    $$
Moreover, 
\begin{equation}
\label{eqn:in}
  \#\{ \pi_{a_1}(E_{a_1}[\infty])\cap \pi_{a_2}(E_{a_2}[\infty])\} = 6 + 4 n\ge 10,
\end{equation}
where $6$ 
is the number of images of common  points of order $4$ (from Equation~\ref{eqn:eqb})
and $4$ stands for the size of $(\Z/2\oplus \Z/2)$-orbit of a point in $\P^1$.
However, it may happen that the inequality in \eqref{eqn:in} is strict. 
 
\begin{exam}
Consider the polynomial $f_5(x ,a )$ defined in \cite[Theorem 18]{BF}). Its roots are exactly
$\pi_a(E_a[5])$.  It has degree $12$ with respect to $x$ and $6$ with respect to $a$.
The polynomial $f_3(x,a)$ has degree $2$ with respect to $a$ and generically has exactly
two solutions $a_1(z),a_2(z)$, for any given $z$.
We want also $f_5(v, a_i(z))= 0$ for some $v$ and $z$.
This is equivalent to $f_5(v, a)$ being divisible by $f_3(z, a)$, as polynomials in $a$.
Writing division with remainder
$$
f_5(v,a)= g(a)f_3(z,a) + C(v,z)a + C'(v,z) 
$$
for some explicit polynomials $C$, and $C'$, which have to vanish.  
This condition is gives an explicit polynomial in $u$, which is divisible by a high power of $u$ and $(u-1)$.
Excluding the trivial solutions $u=0,1$, and substituting $t=u^4$ we obtain the equation  
\begin{eqnarray*}
 & & 32u^{24}+ 1369u^{20}+18812u^{16}+90646u^{12}+18812u^{8}+1369u^{4}+32\\
 & = & 32t^{6}+1369t^{5}+18812t^{4}+90646t^{3}+18812t^{2}+1369t+32\\
 & = & t^{3}\left[32\left(t^{3}+\frac{1}{t^{3}}\right)+1369\left(t^{2}+\frac{1}{t^{2}}\right)+18812\left(t+\frac{1}{t}\right)+90646\right]
 \end{eqnarray*}
 Since $t\neq 0$,  we have
\begin{eqnarray*} 
& =& 32\left(t+\frac{1}{t}\right)^{3}+1369\left(t+\frac{1}{t}\right)^{2}+18716\left(t+\frac{1}{t}\right)+87908\\
&=& 32r^{3}+1369r^{2}+18716r+87908\\
 &=:&  f(r)
\end{eqnarray*}
Computing the discriminant of this cubic polynomial, we find that it has no multiple roots. 
Its solutions give rise to  pairs $u,v$ such that
 for $a_1:=a_1(u),a_2:=a_2(u)$ we have 
 $$
 f_5(v, a_i) =  f_3(u, a_i) =0 
 $$
and hence
$$
\#\{     \pi_{a_1}(E_{a_1}[\infty])\cap \pi_{a_2}(E_{a_2}[\infty])\} \geq 14.
$$
The symmetry of the above equation reduced the problem to a cubic equation with coefficients in $\Q$, 
followed by a quadratic equation. The roots can be expressed in closed form and hence we get explicit 
description for the 24 roots $u$.  
\end{exam}

    The same scheme can be applied to points of higher order.
    Indeed we have a polynomial $f_n(u,x)= 0$ which has
    increasing degree with respect to $u$, and the
    existence of a pair  $u,v$ such that  $f_n(v,x)= 0$ is divisible
    by $f_3(u,x)$ depend on the divisibilty of $f_n(v,x)$ by  $f_3(u,x)$.
    Using long division we obtain two polynomials
    $C_{0,n}( u,v)$ and $C_{1,n}( u,v)$ so that their common
    zeroes $(u,v)$ correspond to pairs $( u,v)$
    with  $f_3(u,x)= 0$ and $f_n(v,x)= 0$ simulaneously.

\begin{exam}
Applying  this scheme to  points of order $3$ and $7$ (or 3 and 11, 3 and 13, 3 and 17)
we obtain that the corresponding resultant has
roots of multiplicity three which implies the existence
of three points $v$ for a given $u$ with
 $f_3(u,x)= 0$ and $f_7(v,x)= 0$ and hence
 $$
  \#\{ \pi_1(E_1[\infty])\cap \pi_2( E_2[\infty])\}  \ge 6 + 16 = 22.
  $$
\end{exam}

 Since we have every reason to expect
 polynomials $C_{0,n}( u,v)$ and $C_{1,n}( u,v)$ to have increasing
 number of intersection points with the growth of $n$ we
 are led to the following conjecture:

 \begin{conj} 
 There is an infinite dense subset
 of points $a\in \P^1$ such that 
 $$
 \pi_a(E_{a}[\infty])\cap \pi_{a_2}(E_{a_2}[\infty]) \ge 22
 $$
 with  
 $$
 \pi_a(E_{a}[3])\cap \pi_{a_2}(E_{a_2}[3]) \neq 0.
 $$
 \end{conj}
 Note that in all such cases the fields $L_{a} = L_{a_2}$.

\section{General Weierstrass families}
\label{sect:gen-fam}

The family of elliptic curves considered in Section~\ref{sect:inter}
is the most promising for obtaining large intersections of torsion points.
In this section, we consider other families where the intersections
tend to be smaller, following \cite{BF}.

We consider elliptic curves $E_a$ with the same 
$$
\pi_a(e_a)=\infty\in \P^1.
$$
These are given by their Weierstrass form
\begin{equation}
\label{eqn:wei}
y^2 = x^3 + a_2x^2+ a_4 x + a_6. 
\end{equation}
Using formulas in, e.g.,  \cite[III, Section 2]{knapp}, we write down
(modified) division polynomials
$f_{n, a}$, whose zeroes are {\em exactly} $\pi_a(E_a[n])$:
$$ 
f_{n,a}(x) =\sum_{0\le r,s,t, r+ 2s +3t \leq d(n)} c_{r,s,t}(n) a_2^r a_4^s a_6^t x^{d(n)- (r+ 2s +3t)},
$$ 
where $d(n)$ and the coefficients $c_{r,s,t}(n)$ 
can be expressed via totient functions 
$J_k(n)$, with $d(n) = J_2(n)/2$, if $n> 2$, and $d(2)=3$ (see \cite{BF}).

\begin{lemm} 
\label{lemm:fn}
Let $E_1, E_2$ be elliptic curves in generalized 
Weierstrass form \eqref{eqn:wei} such that, for some $n>1$ we have
$$
\pi_1(E_1[n])=\pi_2(E_2[n]). 
$$
Then $E_1\simeq E_2$. 
\end{lemm}
 
 \begin{proof}
The statement is trivial for $n=2$. 
For $n > 2$, we have $d(n) \geq 4$, the comparison of division polynomials
implies that the terms
$$
 a_2^r a_4^s a_6^t, \quad r+ 2s +3t \leq 3 
$$ 
must be equal. For 
$$
(r,s,t)=(0,0,1), (0,1,0), (1,0,0)
$$ 
we find equality of coefficients $a_i$ for both curves. 
\end{proof}

  Often, already the existence of nontrivial
  intersections 
 \begin{equation}
 \label{eqn:eee}
 \pi_1(E_1[n]) \cap \pi_2(E_2[n]) \ge 1
 \end{equation}
  leads to the
  isomorphism of curves $E_1,E_2$. 
For example, if both curves are defined over a number field $k$ and
the action of the absolute Galois group $G_k$ on $\pi_1(E_1[n])$
and $\pi(E_2[n])$ is transitive then \eqref{eqn:eee}  implies that $E_1\simeq E_2$.
For many, but not all, $n\in \N,$ the equality of totient functions $J_2(n)= J_2(m)$,  for some $m\in \N$,  implies $n=m$.


\begin{exam}
There exist many tuples $(m,n)$ for which 
$$
J_2(m)=J_2(n)\quad \text{ and } \quad J_1(m)\neq J_1(n).
$$ 
For example, 
$$
J_2(5)=J_2(6)\quad \text{ but } \quad J_1(5) = 4, \quad J_1(6)= 2.
$$ 
We also have 
$$
J_2(35)=J_2(40)= J_2(42),  \text{ while } J_1(35)= 24,  J_1(40)=16,  J_1(42)=12.
$$  
On the other hand, we have 
$$
J_2(15)=J_2(16)=192 \quad \text{ and } \quad J_1(15)=J_1(16)=8.
$$
 \end{exam}

These results indicate a relation of our question
to Serre's conjecture. He considered the action of the Galois group 
on torsion points of an elliptic curve  $E$ defined over a number field $k$.
If $E$ does not have complex multiplication, then 
the image of the absolute Galois group $G_k$ is an open subgroup of $\GL_2(\hat \Z)$, i.e., of finite
index.

\begin{conj}[Serre]
For any number field  $k$ there exists a constant $c=c(k)$ such that for every 
non-CM elliptic curve $E$ 
over $k$ the index of the image of the Galois group $G_k$ in $\GL_2(\hat \Z)$
is smaller than $c$.
\end{conj}

In particular, for $k= \Q$ he conjectured that for primes $\ell \geq 37$ the
image of $G_k$ surjects onto $\PGL_2(\Z_\ell)$.
Thus, modulo Serre's conjecture, our conjecture holds for curves
defined over $\Q$.

\begin{prop} 
Assume that 
$$
\pi_1( E_1[n])=\pi_2( E_2[m]), \quad n\neq m.
$$
Then $k(E[n])$ contains $\Q(\zeta_d)$, where $d = \mathrm{lcm}(m,n)$, 
the least common multiple of $m,n$.
\end{prop}

\begin{proof} 
By Serre,  we have
$$
\Q(\zeta_n)\subset k(E[n]) \quad \text{ and } \quad 
\Q(\zeta_m)\subset k(E[m])
$$ 
as subfields
of index at most $2$. \end{proof}

\begin{coro}
Assume that $k$ does not contain roots of $1$
of order divisible by $n,m$. Then
$k(E[n]), k(E(m))$ contain a cyclotomic subfield
of index at most $2$.
\end{coro}

This provides a strong restriction on intersections of images of torsion points
for elliptic curves over $\Q$, or over more general number fields $k$ with this property.
This yields a restriction on fields $k(E[n])$, 
since $(n,m) > 4$,  for all $(n,m)$ with $J_2(n)= J_2(m)$.

\bibliographystyle{plain}
\bibliography{elliptic}

\end{document}